\documentclass{amsart}
\usepackage{amsmath}
\usepackage{amssymb}
\usepackage{amsthm,amssymb,amsmath,enumerate,graphicx}
\usepackage{caption}

\newtheorem{theorem}{Theorem}[section]
\newtheorem{lemma}[theorem]{Lemma}

\newtheorem{observation}[theorem]{Observation}

\newtheorem{conjecture}[theorem]{Conjecture}

\theoremstyle{remark}
\newtheorem{notation}[theorem]{Notation}
\newtheorem{remark}[theorem]{Remark}
\newtheorem{definition}[theorem]{Definition}
\newtheorem{example}[theorem]{Example}

\newcommand{\cf}{\mathcal F}

\newcommand{\ch}{\mathcal H}

\newcommand{\lijm}{\ell_{j}^M}
\newcommand{\lijw}{\ell_{j}^W}
\newcommand{\eij}{e_{i_j}}
\newcommand{\eik}{e_{i_k}}

\newcommand{\rik}{R_{i_k}}
\newcommand{\rij}{R_{i_j}}
\newcommand{\Rij}{R_{i_j}}
\newcommand{\Rik}{R_{i_k}}

\newcommand{\fij}{F_{i_j}}
\newcommand{\fik}{F_{i_k}}
\newcommand{\Fij}{F_{i_j}}

\newcommand{\tljm}{{\ell}_{j}^M}
\newcommand{\ljm}{{\ell}_{j}^M}

\newcommand{\tljw}{{\ell}_{j}^W}
\newcommand{\ljw}{{\ell}_{j}^W}

\newcommand{\tjm}{t^M_j}

\newcommand{\tmj}{t^M_j}
\newcommand{\Tjm}{T^M_j}
\newcommand{\Tmj}{T^M_j}
\newcommand{\Tkm}{T^M_k}
\newcommand{\tjw}{t^W_j}
\newcommand{\twj}{t^W_j}
\newcommand{\Tjw}{T^W_j}
\newcommand{\Twj}{T^W_j}

\newcommand{\btjw}{\bar{T}^W_j}

\newcommand{\btjm}{\bar{T}^M_j}

\newcommand{\skipjm}{{SKIP}^M_j}

\newcommand{\skipmj}{{SKIP}^M_j}

\newcommand{\skipkm}{{SKIP}^M_k}
\newcommand{\skipjw}{{SKIP}^W_j}
\newcommand{\skipwj}{{SKIP}^W_j}
\newcommand{\tfij}{\tilde{F}_{i_j}}

\begin{document}

\makeatletter

\makeatother
\author{Ron Aharoni}
\address{Department of Mathematics\\ Technion}
\email[Ron Aharoni]{raharoni@gmail.com}
\thanks{\noindent The research of the first author was
supported by BSF grant no. $2006099$, by GIF grant no. I
$-879-124.6/2005$, by the Technion's research promotion fund, and by
the Discont Bank chair.}

\author{David Howard}
\address{Department of Mathematics\\ Colgate University}
\email[David Howard]{dmhoward@colgate.edu}
\thanks{\noindent The research of the second author was
supported by BSF grant no. $2006099$, and by ISF grants Nos.
$779/08$, $859/08$ and $938/06$.
 }

\title{A rainbow $r$-partite version of the Erd\H{o}s-Ko-Rado theorem}

\begin{abstract}
Let $f(n,r,k)$ be the minimal number such that every hypergraph larger than $f(n,r,k)$ contained in $\binom{[n]}{r}$
contains a matching of size $k$,  and let $g(n,r,k)$ be the minimal number such that every hypergraph larger than $g(n,r,k)$
contained in the $r$-partite $r$-graph $[n]^{r}$ contains a matching of size $k$. The Erd\H{o}s-Ko-Rado theorem states
that  $f(n,r,2)=\binom{n-1}{r-1}$~~($r \le \frac{n}{2}$) and it is easy to show that $g(n,r,k)=(k-1)n^{r-1}$.

The conjecture inspiring this paper is that if
$F_1,F_2,\ldots,F_k\subseteq \binom{[n]}{r}$ are of size larger than
$f(n,r,k)$ or $F_1,F_2,\ldots,F_k\subseteq [n]^{r}$ are of size
larger than $g(n,r,k)$ then there exists a rainbow matching, i.e. a
choice of disjoint edges $f_i \in F_i$.  In this paper we deal
mainly with the second part of the conjecture, and prove it for
$r\le 3$.  \vspace{.1cm}

We also prove that for every $r$ and $k$ there exists $n_0=n_0(r,k)$
such that the $r$-partite version of the conjecture is true for
$n>n_0$.

 \end{abstract}

\maketitle

\section{Motivation}
\subsection{The Erd\H{o}s-Ko-Rado theorem and rainbow matchings}
A {\em matching} is a collection of disjoint sets. As is customary,
we write $[n]$ for the generic set of size $n$, $\{1, 2, \ldots,
n\}$. We denote by $\binom{[n]}{r}$ the set of subsets (also called
``edges'') of size $r$ of $[n]$.
 The largest size of a matching in a hypergraph
$H$ is denoted by $\nu(H)$. The famous Erd\H{o}s-Ko-Rado (EKR)
theorem \cite{ekr} states that if $r\le \frac{n}{2}$ and a hypergraph $H
\subseteq \binom{[n]}{r}$ has more than $\binom{n-1}{r-1}$ edges,
then $\nu(H)
>1$.  This has been extended in more
than one way to pairs of hypergraphs. For example, in \cite{mt,
pyber} the following was proved:
\begin{theorem}\label{mt}
If $r \le \frac{n}{2}$, and $H_1,H_2\subseteq \binom{[n]}{r}$ satisfy
$|H_1||H_2|>\binom{n-1}{r-1}^2$ (in particular if
 $|H_i|>\binom{n-1}{r-1},~i=1,2$),
  then there
exist disjoint edges, $e_1 \in H_1,~e_2 \in H_2$.
\end{theorem}

In \cite{mt} this was extended to hypergraphs of different uniformities.

It is natural to try to extend the EKR theorem to more than two hypergraphs. The
relevant notion is that of ``rainbow matchings".

\begin{definition}
 Let $\cf=(F_i \mid~~1 \le i \le k)$ be a collection of hypergraphs. A
 choice of disjoint edges, one from each $F_i$, is called a {\em
 rainbow matching} for $\cf$.
 \end{definition}

\begin{notation}
For $n,r,k$ satisfying $r \le \frac{n}{2}$ we denote by $f(n,r,k)$
the smallest number such that $\nu(H)\ge k$ for every $H \subseteq
\binom{n}{r}$ larger than $f(n,r,k)$. \end{notation}

The value of $f(n,r,k)$ is known for large enough $n$:

\begin{theorem} \cite{erdosplus} For every $r,k$ there exists $n_0=n_0(r,k)$ such that for every $n \ge n_0$:

$$f(n,r,k)=\binom{n}{r}-\binom{n-k+1}{r}$$
\end{theorem}

The following  was proved in \cite{frankl}:

\begin{theorem}
If $n \ge kr$ then $f(n,r,k) \le (k-1)\binom{n-1}{r-1}$.
\end{theorem}

A rainbow version of this theorem was proved in \cite{hls}:

\begin{theorem}\cite{hls}
If $F_1, \ldots, F_k$ are hypergraphs, where $F_i$ is $r_i$-uniform
and $n \ge \sum_{i \le k}r_i$ and $|F_i|>(k-1)\binom{n-1}{r_i-1}$
then the family $(F_1, \ldots, F_k)$ has a rainbow matching.
\end{theorem}

It is a natural guess that Theorem \ref{mt} can be extended to
general $k$, as follows:

\begin{conjecture}\label{rainbowgeneral}
Let $\cf=(F_1, \ldots, F_k)$ be a system of hypergraphs contained in
$\binom{[n]}{r}$. If $|F_i|>f(n,r,k)$  for all $i \le k$ then $\cf$ has a
rainbow matching. \end{conjecture}

In \cite{hy} the case $r=3$ of Conjecture \ref{rainbowgeneral} is
solved for $n \ge 4k-1$. In Section \ref{shifting} we shall present
a proof by Meshulam for the $r=2$ case of this conjecture.

\subsection{The $r$-partite case}

An $r$-uniform hypergraph $H$ is called {\em $r$-partite} if $V(H)$
is partitioned into sets $V_1,\ldots, V_r$, called the {\em sides}
of $H$, and each edge meets every $V_i$ in precisely one vertex. If
all sides are of the same size $n$, $H$ is called {\em
$n$-balanced}. The complete $n$-balanced $r$-partite hypergraph can be identified with $[n]^r$.

\begin{remark} Let $H$ be an $r$-partite hypergraph, and let $V_i$ be one of its sides. There exists a matching in $H$ covering
$V_i$ if and only if  there exists a  rainbow matching of the family
 $H_v$, $v \in V_i$, where $H_v$ is the hypergraph
consisting of the $(r-1)$-edges incident with  $v$. \end{remark}


Conditions of different types are known for the existence of rainbow
matchings. For example, in \cite{ah, haxell} a sufficient condition
was formulated in terms of  domination in the line graph of
$\bigcup_{i \in I}F_i$ and in terms of $\nu(\bigcup_{i \in
K\subseteq I}F_i)$~($I$ ranging over all subsets of $[k]$).

Here we shall be interested in
 conditions formulated in terms of the sizes of the hypergraphs.

\begin{observation}\label{sizecond}
If $F$  is a set of edges in an  $n$-balanced $r$-partite hypergraph
and $|F|>(k-1)n^{r-1}$ then $\nu(F) \ge k$.
\end{observation}

\begin{proof}
The complete $n$-balanced $r$-partite hypergraph $[n]^r$ can be
decomposed into $n^{r-1}$ perfect matchings $M_i$, each of size $n$.
Writing $F =\bigcup_{i \le n^{r-1}} (F \cap M_i)$ shows that at least one of
the matchings $F \cap M_i$ has size $k$ or more.
\end{proof}

The $r$-partite analogue of Conjecture \ref{rainbowgeneral} is:

\begin{conjecture}\label{sizecondition}
If $\cf=(F_1,F_2, \ldots,F_k)$ is a set of sets of edges in an
$n$-balanced $r$-partite hypergraph  and $|F_i| > (k-1)n^{r-1}$ for
all $ i \le k$ then $\cf$ has a rainbow matching.
\end{conjecture}

In \cite{ahoward} this is proved for $k=2$.

The following result, stating the case $r=2$, will be subsumed by
later results, but this case has a short proof of its own:

\begin{theorem}\label{mainbipartite}
If $\cf=(F_1,F_2, \ldots,F_k)$ is a set of sets of edges in an
$n$-balanced bipartite graph  and $|F_i| > (k-1)n$ for all $ i \le
k$ then $\cf$ has a rainbow matching.
\end{theorem}

\begin{proof}

Denote the sides of the bipartite graph $M$ and $W$. Since $\sum_{v
\in M} d_{F_1}(v)=|F_1|>(k-1)n$, there exists
 a vertex $v_1 \in M$ such that $d_{F_1}(v_1)\ge k$. Let
 $F'_2=F_2-v_1$. Since $d_{F_2}(v_1)\le n$, we have $|F'_2| >
 (k-2)n$, and hence there exists a vertex $v_2 \neq v_1$ such that $d_{F_2}(v_2)\ge
 k-1$. Continuing this way we obtain a sequence $v_1, \ldots, v_k$
 of distinct vertices in $M$, satisfying $d_{F_i}(v_i)> k-i$.
 Since $d_{F_k}(v_k)>0$ there exists an edge $e_k \in F_k$ containing
 $v_k$.
Since $d_{F_{k-1}}(v_{k-1})>1$ there exists an edge $e_{k-1} \in
F_{k-1}$ containing
 $v_{k-1}$ and missing $e_k$.
 Since $d_{F_{k-2}}(v_{k-2})>2$ there exists an edge $e_{k-2} \in
F_{k-2}$ containing
 $v_{k-2}$ and missing $e_k$ and $e_{k-1}$. Continuing this way, we
 construct a rainbow matching $e_1, \ldots,e_k$ for $\cf$.
 \end{proof}

We shall prove:

\begin{theorem}\label{generalr3}
Conjecture \ref{sizecondition} is true for $r=3$.
\end{theorem}

\section{Shifting} \label{shifting}
 The proof in \cite{ekr} uses an operation called ``shifting'', that has since become  a main tool in the area. It is an operation on a
hypergraph $H$, defined with respect to a specific linear ordering
``$<$" on its vertices. For $x<y$ in $V(H)$ define  $s_{xy}(e)=e
\cup {x} \setminus \{y\}$ if $x \not \in e$ and $y \in e$, provided
 $e \cup {x} \setminus \{y\} \not \in H$; otherwise let $s_{xy}(e)=e$.
  We also write $s_{xy}(H)=\{s_{xy}(e) \mid e \in H\}$. If
$s_{xy}(H)=H$ for every pair $x<y$ then $H$ is said to be {\em
shifted}.

Given an $r$-partite hypergraph $G$ with sides $M$ and $W$, and
linear orders on its sides, an {\em $r$-partite shifting} is a
shifting $s_{xy}$ where $x$ and $y$ belong to the same side. $G$ is
said to be {\em $r$-partitely shifted} if $s_{xy}(H)=H$ for all
pairs $x<y$ that belong to the same side.

Given a collection $\ch=(H_i,~~i \in I)$ of hypergraphs, we write
$s_{xy}(\ch)$ for $(s_{xy}(H_i),~~i \in I)$.

\begin{remark}\label{shifted}
Define a partial order on pairs of vertices  by $(v_i,v_j) \le
(v_k,v_\ell)$ if $i \le k$ and $j \le \ell$. Write $(v_i,v_j) <
(v_k,v_\ell)$ if $(v_i,v_j) \le (v_k,v_\ell)$ and $(v_i,v_j) \neq
(v_k,v_\ell)$. A set $F$ being shifted is equivalent to its being
closed downward in this order, which in turn is equivalent to the
fact that the complement of $F$ is closed upward.
\end{remark}

As observed in \cite{erdosplus} (see also \cite{af}) shifting does not increase the matching number of a hypergraph.
This can be generalized to rainbow matchings:

\begin{lemma}\label{shifting}
Let $\cf=(F_i \mid~~i\in I)$ be a collection of hypergraphs, sharing the same
linearly ordered ground set $V$, and let $x<y$ be elements of $V$. If
$s_{xy}(\cf)$ has a rainbow matching, then so does $\cf$.
\end{lemma}

\begin{proof}
Let $s_{xy}(e_i),~~i \in I$, be a rainbow matching for
$s_{xy}(\cf)$. There is at most one $i$ such that $x \in e_i$, say
$e_i=a \cup\{x\}$ (where $a$ is a set).


 If there
is no edge $e_s$ containing $y$, then replacing $e_i$ by $a
\cup\{y\}$ as a representative of $F_i$, leaving all other $e_s$ as
they are, results in a rainbow matching for $\cf$. If there is an
edge $e_s$ containing $y$, say $e_s=b \cup\{y\}$, then there exists
an edge $b \cup\{x\} \in F_s$ (otherwise the edge $e_s$ would have
been shifted to $b \cup\{x\}$). Replacing then $e_i$ by $a
\cup\{y\}$ and $e_s$ by $b \cup\{x\}$ results in a rainbow matching
for $\cf$.
\end{proof}

\section{Conjecture \ref{rainbowgeneral} for $r=2$}

In \cite{erdosplus} the value of $f(n,2,k)$ was determined for all
$k$:

\begin{theorem}
$f(n,2,k)=\max(\binom{2k-1}{2},(k-1)(n-1)-\binom{k-1}{2})$.
\end{theorem}

In \cite{af} this result was given a short proof, using shifting.
Meshulam \cite{meshulam} noticed that this proof yields also
Conjecture \ref{rainbowgeneral} for $r=2$:

\begin{theorem}\label{af}
Let $\cf=(F_i, ~~1 \le i \le k)$ be a collection of subsets of
$E(K_n)$. If $|F_i|>
\max(\binom{2k-1}{2},(k-1)(n-1)-\binom{k-1}{2})$ for all $i \le k$
then $\cf$ has a rainbow matching.
\end{theorem}

\begin{proof}
Enumerate the vertices of $K_n$ as $v_1,v_2,\ldots, v_n$. By Lemma
\ref{shifting} we may assume that all $F_i$'s are shifted with
respect to this enumeration. For each $i \le k$ let
$e_i=(v_i,v_{2k-i+1})$. We claim that the sequence $e_i$ is a
rainbow matching for $\cf$. Assuming negation, there exists $i$ such
that $e_i \not \in F_i$. Since $F_i$ is shifted, every edge
$(v_p,v_q)$ in $F_i$, where $p<q$, satisfies\\

(P)~~ $p<i$ or $q<2k-i+1$.
\\

The number of  pairs satisfying $p<i$ is
$(i-1)(n-1)-\binom{i-1}{2}$. The number of pairs satisfying $p\ge i$
and $q<2k-i+1$ is $\binom{2k-2i+1}{2}$, so

$$|F_i| \le (i-1)(n-1)-\binom{i-1}{2} + \binom{2k-2i+1}{2}$$

This is a convex quadratic expression in $i$, attaining its maximum
either at $i=1$ (in which case $|F_i| \le \binom{2k-1}{2}$) or at
$i=k$ (in which case $|F_i| \le (k-1)(n-1)-\binom{k-1}{2}$). In both
cases we get a contradiction to the assumption on $|F_i|$.
\end{proof}

\section{A Hall-type size condition for rainbow matchings in bipartite graphs}

In this section we prove a result on the existence of rainbow
matchings for a collection of bipartite graphs, all sharing the same
vertex set and bipartition, that will be later used for the proof of
Theorem \ref{generalr3}. This condition is not formulated in terms
of the sizes of the individual graphs, but (a bit reminiscent of the
condition in Hall's theorem) in terms of the sizes of subsets of the
collection of graphs.

\begin{theorem}\label{maintheorem}
Let  $F_i,~~i \le k$ be subsets of $E(K_{n,n})$.
 If

 \begin{equation}\label{condition}
 \sum_{i \in I}|F_i| > n|I|(|I|-1) \mbox{ for every} ~~I \subseteq [k]
 \end{equation}
 then the system $\cf=(F_1, \ldots ,F_k)$ has  a rainbow matching.
  \end{theorem}

Sharpness of this bound is shown by the example of $k$ sets $F_i$,
 each consisting of all edges incident with a set  of $k-1$ vertices in one side of the bipartite graph.
The analogous result for $r=1$ can be proved directly, or using
Hall's theorem. For $r \ge 3$ the analogous result,
suggested by the same example, is that if
$\sum_{i \in I}|F_i| > n^2|I|(|I|-1)$ for all $I$ then the system $(F_1, \ldots ,F_n)$ has a
rainbow matching. But this is false, as shown by the pair $F_1, F_2$ in
which $F_1$ consists of a single edge and $F_2$ the set of all edges meeting this edge.
Then $|F_2|=n^3-(n-1)^2$, $|F_1|+|F_2|=3n^2-3n$, which for $n>3$ is larger than $2n^2$, and there is no rainbow matching.
It is not clear what is the right condition for general $r$.

\subsection{ An algorithm}

The proof of Theorem \ref{maintheorem} is algorithmic. As before, we
assume that each side of the bipartite graph is linearly ordered,
say $M=(m_1 <m_2 <
 \ldots< m_n)$ and $W=(w_1 <w_2<  \ldots < w_n)$.

 \begin{definition}\label{crossing} Two edges $e,f$ are said to be {\em parallel} if the order between their $M$ vertices is the same as the order between their $W$ vertices. If in this case the vertices of $e$ precede those of $f$, we write $e<f$. Non-parallel edges are said to be {\em crossing}.
 \end{definition}

 By Lemma \ref{shifting}, we may assume that all $F_i$
 are bipartitely shifted with respect to the given orders.

Order the sets $F_i$ by their sizes,
\begin{equation}\label{monotone}
|F_1| \le |F_2| \le \ldots \le |F_k|
\end{equation}

We  choose inductively edges $e_i \in F_i$. As $e_1$ we choose a
longest edge $(m_{c(1)},w_{d(1)})$ in $F_1$,
where the length of an edge $(m_p,w_q)$ in this case is $|q-p|$. By the shiftedness of
$F_1$, either $c(1)=1$ or $d(1)=1$.

Suppose  that  $e_1 \in F_1, ~e_2 \in F_2, \ldots, e_{t-1} \in
F_{t-1}$ have been chosen. Let $Z_t =\bigcup_{j < t} e_j$. Let $a_t$
the first index such that $m_{a_t} \not \in Z_t$, and $b_t$ be the
first index such that $w_{b_t} \not \in Z_t$. Let $R_t =\{m_1,
\ldots ,m_{a_t-1}\} \cup \{w_1, \ldots ,w_{b_t-1}\}$ ($R_1$ is the
empty set).

Let $\tilde{F_t}=F_t[V \setminus Z_t]$ (the set of edges in $F_t$ not meeting $Z_t$).
Define the {\em length} of an edge $(m_p,w_q) \in {F_t}$ as $|(q-b_t)-(p-a_t)|$.
Assuming that $\tilde{F_t}\neq \emptyset$, choose $e_t$ to be a longest edge in $\tilde{F_t}$.
Since $F_t$ is shifted, $e_t$ must contain either $m_{a_t}$ or $w_{b_t}$.

The fact that $e_t\in \tilde{F_t}$ implies inductively that the edges
$e_i,~i \le t$, form a matching. The
proof will be complete if we show that $\tilde{F_t} \neq \emptyset$
for all $t \le n$. \\

The following example illustrates the way the algorithm proceeds. In
it the inequalities of \eqref{condition} are violated, and indeed
the algorithm fails, although in fact there is a rainbow matching.

\begin{example}\label{stealexample} Let $q<n$. Let
$F_1=\{m_cw_d \mid c,d \le q\}$, and let $F_2=F_3= \ldots= F_{q+1}=\{m_cw_d \mid c\le q,~d \le n\} \cup \{m_cw_1 \mid c \le n\}$.
\\
Here $|F_i|=(q+1)n-q$ for all $1<i\le q+1$, and hence $\sum_{i\le {q+1}}|F_i|=q^2+q[(q+1)n-q]=q(q+1)n$, so in this case \eqref{condition} is violated, with equality replacing strict inequality. Indeed, as we shall see, the algorithm fails. Yet, there exists a rainbow matching: $F_1$ is represented by $m_1w_q$, $F_2$ is represented by $m_nw_1$, and  $F_i$ is represented by $m_{i-1}w_{n-i+2}$ for $i>2$.

 Here is how the algorithm goes (we are assuming below that $q\ge 3$):

$$R_1=\emptyset,~e_1=m_qw_1,~R_2=\{w_1\},~e_2=m_1w_n,~R_3=\{m_1,w_1\}, ~e_3=m_2w_{n-1}, \ldots, $$ $$e_q=m_{q-1}w_{n-q+2},~R_{q+1}=\{w_1\} \cup \{m_c \mid c\le q\}.$$

After the choice of $e_q$ there is no possible choice for $e_{q+1}$ and the algorithm halts. Note that in the first step it was also legitimate  to choose $m_1w_q$, which would lead to a rainbow matching.\\

\end{example}

Let us now return to the proof. Suppose, for contradiction, that $\tilde{F}_m = \emptyset$ for some
$m  \le n$. We shall show that this entails a violation of \eqref{condition}, for $I=[m]$.

For each $i <m$ let $c(i),~d(i)$ be such that  $e_i
=(m_{c(i)},w_{d(i)})$. As already noted, by shiftedness either
$c(i)=a_i$ or $d(i)=b_i$. We direct $e_i$, calling one of its
endpoints ``tail'' and the other ``head'', as follows. If $c(i)=a_i$
we call $m_{a_i}$ the {\em tail} of $e_i$, and $w_{d(i)}$ its {\em
head}. Otherwise, we call $w_{d(i)}$ the tail, and $m_{c(i)}$ the
head. We write $tail(e_i)$ for the tail, and $head(e_i)$ for the
head. We clearly have:
\begin{observation}\label{tail}
If $i<j$ then $tail(e_i) \in R_j$.

\end{observation}

\subsection{Short edges}
We call the edges $e_i$ contained in $R_m$ {\em short} and an edge
not contained in $R_m$ {\em long}. Let $\eij, ~~j <p$, be the short
edges, where $i_1<i_2<\ldots< i_{p-1}$ (so, there are $p-1$ short
edges). Define $i_0=0$ and $i_p=m$. To understand the significance
of short edges, note that if there are no short edges
 then $|R_m|=m-1$. Since
$\tilde{F}_m =\emptyset$,   the set $R_m$ is a cover for $F_m$, and hence
 $|F_m| \le (m-1)n$. By \eqref{monotone} this implies that $\sum_{i \le m}|F_i|
\le m(m-1)n$, contradicting the assumption of the theorem.\\

\begin{example}
In Example \ref{stealexample} there is only one short edge, $e_1$.
\end{example}

For $j<p$ let $\ell_j^W$ be the length of the longest edge in
$\tfij$ containing $m_{a_{i_j}}$ and let $\ell_j^M$ be the length of
the longest edge in $\tfij$ containing $w_{b_{i_j}}$. Let
$SKIP^M_j=\{m_{a_{i_j}}, m_{a_{i_j}+1}, \ldots, m_{a_{i_j}+\ljm}\}$
and
 $SKIP^W_j=\{w_{b_{i_j}}, w_{b_{i_j}+1}, \ldots, w_{b_{i_j}+\ljw}\}$.

We  denote by $T_j^M$ (resp. $T_j^W$) the longest contiguous stretch of vertices in $Z_{i_j} \cap M$ (resp. $Z_{i_j} \cap W$)  starting right after $\skipjm$ (resp. $\skipjw$),
and let $t_j^M=|T_j^M|$, $t_j^W=|T_j^W|$. See Figures 1 and 2.

\begin{figure}
\centering
\begin{minipage}{.5\textwidth}
  \centering
\vspace{-1.33cm}
  \includegraphics[width=.75\linewidth]{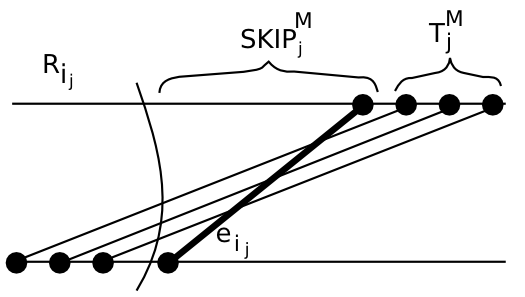}
\vspace{.66cm}
  \captionof{figure}{$\skipmj,~
  \Tmj$ and $\eij$}
  \label{fig:1}
\end{minipage}%
\begin{minipage}{.5\textwidth}
  \centering
  \includegraphics[width=.75\linewidth]{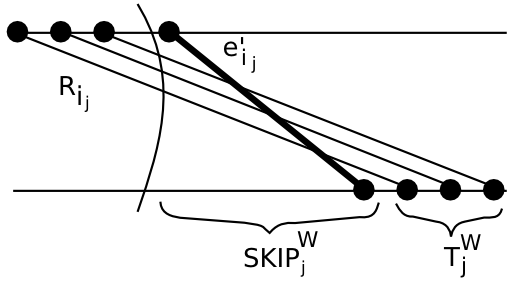}
    \captionof{figure}{$\skipwj$ and   $\Twj$. Here $e_{i_j}'$ is the longest edge in $\tfij$ starting at $a_i$.}
  \label{fig:2}
\end{minipage}
\end{figure}

\section{Bounding $\sum |F_i|$ from above}

\subsection{ A toy case - one short edge}
Our aim is  now to delve into calculations showing that under the negation
assumption $\sum_{i \le m}|F_i|< nm(m-1)$. To demonstrate  the type
of arguments involved in the general proof, let us  consider
separately the  case in which there is only one short edge, say
$e_{i}$. It may be worth following the arguments in Example \ref{stealexample}, in which as mentioned above there is only one short edge.

Recall that either $c(i)=a_i$ or $d(i)=b_i$, and without loss of
generality assume the latter, implying that $d(i)=\min\{j \mid w_j
\not \in R_i\}$.

 Write $\ell$ for $\ell^M_1$, namely the length of $e_i$ (to understand the subscript $1$ in $\ell^M_1$ remember that $i_1=i$).
The edge $e_i$
skips $\ell$ vertices in $R_m$, each being matched by some edge
$e_j,~~i<j<m$, and hence $\ell \le m-i$.

Clearly, $|R_m|=m$, and since $R_m$ is a cover for $F_m$ it
follows that $|F_m| \le mn$. But in this calculation each of the
$\ell$ edges  $(m_{c(i)},w_j)$ for  $j=b_i,b_i+1,\dots,
b_i+\ell-1$, being contained in $R_m$,  is counted twice, from the
direction of $m_{c(i)}$ and from the direction of $w_j$. Thus we know
that:
$$|F_m| \le mn - \ell$$

Since no edge $e_q,~~q<i$, satisfies $e_q<e_i$, we have $|R_i|=i-1$,
and  the number of edges in $F_i$ incident with $R_i$ is thus at
most $(i-1)n$, and by the definition of $\ell$ we have $|F_i| \le
(i-1)n+\ell^2.$ Hence

$$\sum_{q \le m}|F_q| \le i|F_i|+(m-i)|F_m| \le
i((i-1)n+\ell^2)+(m-i)(mn-\ell).$$

Hence

$$m(m-1)n-\sum_{q \le m}|F_q| \ge
m(m-1)n-[i((i-1)n+\ell^2)+(m-i)(mn-\ell)]=(i-1)(m-i)n+(m-i)\ell -
i\ell^2$$
$$=[(i-1)(m-i)n-(i-1)\ell^2]+[(m-i)\ell-\ell^2]$$

Since $\ell \le m-i$ and $\ell \le n$ both bracketed terms are
non-negative, so $m(m-1)n-\sum_{q \le m}|F_q| \ge
0$, reaching the desired contradiction.

\subsection{ Using the short edges as landmarks and a first point of reference}

Let us now turn to the proof of the general case. For $1 \le j \le p-1$ write $s_j=i_j-i_{j-1}$ and let
 $S_j=\{i_{j-1}+1,i_{j-1}+2,\ldots,i_j\}$, so that $|S_j|=s_j$.

By (\ref{monotone}) $|F_k|\le |F_{i_j}|$ for every $k \in S_j$, and hence

\begin{equation}\label{sum}
\sum_{k \le m}|F_k| \le \sum_{j \le p}s_j|\fij|
\end{equation}

The vertices in $\rij$ are  of degree at most $n$, and hence the
number of edges in $\fij$ incident with $\rij$ is at most $n|\rij|$.
We use $n|\rij|$ as a baseline estimate on $|\fij|$. In this
estimate we are ignoring the edges of $\fij$ not incident with
$R_{i_j}$, and also the double counting of edges.

 If there are no
short edges then  $|R_m|=m-1$, and hence $|F_m|\le n|R_m| =(m-1)n$.
Since $|F_i|\le |F_m|$ for all $i\le m$, we have $\sum_{i \le
m}|F_i| \le m(m-1)n$, a contradiction.
 We shall use this calculation as a first point of reference, and to get the real quantities we shall measure the deviations from the estimate
  $|F_i|=(m-1)n$.

The existence of short edges affects the estimate of $\sum_{j \le
p}s_j|\Fij|$ in two ways - adding something to it, and deducting
something.
  The first we call ``loss'', since it takes us further away from the desired contradiction,
and the second  is called  ``gain''. We shall associate a gain $G_j$ and a  loss $L_j$ with each short edge $\eij$,
and we shall show that $G_j \ge L_j$ for every $j\le p$.  Note that our calculation is not uniquely determined,
since adding  the same number to  $G_j$ and to $L_j$ does not change the total balance.

 Clearly, $|\rij|$
 is $i_j-1$, plus the number of short  edges contained in $\rij$. Compared with the  estimate $|\rij|=m-1$ above, the estimate
 $|\rij|=i_j-1$
gives a gain of $m-i_j$ on $|\rij|$, yielding a gain of $n(m-i_j)$ on the estimate $n|\rij|$ of $|\fij|$, which yields a total
gain of $$s_j(m-i_j)n$$ in \eqref{sum}.
In order to obtain an estimate serving as a second point of reference, we assume that $e_{i_j}\subseteq R_{i_k}$ for all $k>j$.
This  entails a loss of $s_kn$ in \eqref{sum} for each such $k$, so altogether there is a loss of

$$
n(s_{j+1}+s_{j+2}+\ldots +s_p)=n(m-i_j).
$$

So, the net gain with respect to the baseline estimate   is so far $s_j(m-i_j)n - n(m-i_j)=(s_j-1)(m-i_j)n$. Writing

\begin{equation}\label{firstnetgain}
G^{BASIC}_j=(s_j-1)(m-i_j)n
\end{equation}

we can use $G^{BASIC}_j$ as a baseline gain.

\subsection{The loss on edges outside $\rij$}
In the above calculation there is an overoptimistic assumption: that
all edges in $\fij$ are  incident with $\rij$. In fact this  is
false for all $j<p$. By shiftedness and the definition of $\lijm,~
\lijw,~ T_j^M$ and $T_j^W$ there can be at most
$(\lijm+t^M_j)(\lijw+\tjw)$ edges that are not incident with $\rij$.

  Remembering that  $|F_{i_j}|$ is multiplied by $s_j$ in \eqref{sum}, this
entails a possible loss of:

\begin{equation} \label{sjlambdaj}
L_j:=s_j(\lijm+t^M_j)(\lijw+\tjw)
\end{equation}

This is the only loss we encounter, besides the loss incurred by
short edges being contained in sets $\rij$, that has already been
subsumed in $G^{BASIC}_j$.

\subsection{Two types of regains}
We shall use two types of regains, related to two ways in which $|\fij|$
was overestimated.
\begin{enumerate}
\item
Gains on procrastination. If $k<j$ we were assuming above that $\eik
\subseteq R_{i_j}$. When this does not happen we say that $j$
{\em procrastinates with respect to $k$} (meaning that $\rij$ is late to capture the edge $\eik$), and then $|\Rij|$ was
overestimated by $1$, giving rise to a gain of $n$ in $|\fij|$, and
to a gain of $s_jn$ in the total sum.
\item
Gains on double counting. In the basic estimate $n|\rij|$ of the number
of edges incident with $\rij$ there is an overestimate of $1$ on
each pair $(u,v)$ of vertices in $\rij$, where $u \in M$ and $v \in
W$. This entitles us to a gain of $s_j$ in the total sum.
\end{enumerate}

\subsection{A first gain on double counting, and a first offset with $L_j$}\hfill

Without loss of generality we may (and will) assume that $\ljm \ge \ljw$, and
that $tail(\eij) \in W$. Then

\begin{equation}\label{fulllj}
L_j \le s_j[\ljm(\ljw+\tjm+\tjw)+\tjm\tjw]
\end{equation}

Here we turn to our first gain on double counting.
Let $E_j=\{e_i \mid i < i_j\}$ be the partial rainbow matching chosen so far. Let $\bar{T}^M_j=E_j[T^W_j]$ (namely the set of vertices in $M$ matched by $E_j$ to $\Twj$), and let $\bar{T}^W_j=E_j[T^M_j]$.
The edges of $\btjm \times \btjw$ were counted twice in the estimate $n\rij$ of the number of edges incident with $\rij$. This entitles us to a gain of $\tjm\tjw$ in the calculation of $|\fij|$, which results in a regain of $s_j\tjm\tjw$ in \eqref{sum}. Offsetting this with part of $L_j$ as appearing in \eqref{fulllj}, and writing

\begin{equation}\label{defoflambdaj} \lambda_j:=\ljw+\tjm+\tjw, \end{equation}

  this leaves us with a loss of at  most

 \begin{equation}\label{lrj} L^r_j:=s_j\ljm\lambda_j \end{equation}

   The superscript $r$ stands for ``remaining''. This loss should be offset by $G^{BASIC}_j$ and by other gains.

The following is clear:

 \begin{observation}\label{lambdalessthann}
$\lambda_j <n$.
\end{observation}

\section{Gains associated with vertices in $SKIP^M_j$}

\subsection{Six types of vertices in $\skipmj$ and the regains associated with them}

\begin{notation}\label{iv}
For $v \in R_m$ let $i(v)$ be the index $i$ for which $v \in e_i$,
and let $k(v)$ be the index $k <p$ such that $i(v) \in S_k$.
\end{notation}

\begin{notation}
Let $\Sigma_j$ be the set of short edges contained in $\rij$, and
let $|\Sigma_j|=\sigma_j$. Also let $M_j=M \cap \rij$ and $\mu_j=|M_j|$.
\end{notation}

\begin{notation}\label{omega}
Let $\omega=\omega(j)=\min(k: \rik \supseteq \eij)$.
\end{notation}

\begin{lemma}\label{lambdasmall}
If $k<j$ and $head(e_{i_k}) \in \skipmj$ then $\lambda_k <
\mu_j+\ell_j^M$.
\end{lemma}
\begin{proof}
This follows from the fact that

$$E_k[T_k^W] \cup \skipkm \cup \Tkm \subsetneqq (R_j \cap M) \cup
\skipjm$$

and on both sides the terms of the union are disjoint. The reason
for the strict containment is that $head(e_{i_j})$ belongs to the
right hand side and not to the left. In fact, the strict inequality
in the lemma will not be used, it is only mentioned for
clarification.
\end{proof}

\begin{lemma}\label{lambdalessthanmu}
 $\lambda_j \le \mu_j+\ljm +\tjw$.
\end{lemma}

This follows from the fact that $\ljw \le \ljm$ and
$E_j[\Tjw]\subseteq M_j$.

\begin{lemma}\label{romega}
$R_{i_\omega} \supseteq \Tjm$
\end{lemma}

\begin{proof} By the definition of $\Tjm$ the vertex $head(\eij)$ is adjacent to its first element, so the initial segment of $head(\eij)$ in $M$, together with $\Tjm$, is an  interval contained in $Z_{i_j}$.
Applying the definition of
 $R_{i_\omega}$ yields the lemma. \end{proof}

We write $ L^r_j$ as a sum: $$L^r_j=L^a_j + L^b_j$$
where
 \begin{equation} \label{splitlr}
 L^a_j=\ljm (s_j-1)\lambda_j,~~~L^b_j=\ljm\lambda_j.
 \end{equation}

The expression \eqref{firstnetgain} for $G^{BASIC}_j$ explains why  this splitting will be useful: $G^{BASIC}_j$ will count towards offsetting $L^a_j$.

We shall have two ``baskets'' of gains for each $j$, which we shall call $G_j^a$ (intended to compensate for $L_j^a$) and $G_j^b$ (intended to compensate for $L_j^b$).
To compensate for $L^b_j$, we need to assign to
each of the $\ljm$ vertices  in $\skipmj$  a gain of at least $\lambda_j$, which is given to $G^b_j$.

For the purpose of bookkeeping, we gather the vertices of $\skipmj$
into six types, according to the conditions they satisfy. Vertices
of types (2b) and (3) below will give rise to regains on double
counting, while all other types will give rise to regains on
procrastination. In all these cases a gain is given to $G^b_h$,
where $h$ is the smaller of $j$ and $k$, namely if $j<k=k(v)$ at
least $\lambda_j$ is given to $G^b_j$, and if $k=k(v)<j$ at least
$\lambda_k$ is given to $G^b_k$.

In two of the cases, namely (2ai) and (1),  the gain will be split
between the two indices. The part
given to the larger index will go to $G^a$ of that index.

Here are the explicit classification and the rules by which gains
are shared. The regain of $\lambda_j$ for each vertex $v \in
\skipmj$ will be apparent in each of the cases, while the regains
accumulating to $G^a_j$ will be collected at the end.

\begin{enumerate}

\item
$k(v) <j$, implying that
 $v=head(\eik)$ (see Figure 3).  In this case $j$ procrastinates with respect to $k$, entitling us to a gain of $n$ on $|\fij|$, and $s_jn$ in total.
This gain we split between $G^b_j$,   $G^a_k$ and $G^b_k$, as follows: $G^a_k$ gets
$(s_j-1)(\mu_j+\ljm)$,  $G^a_j$ gets $(s_j-1)(n-\mu_j-\ljm)$ and $G^b_j$ gets $n$.\\

\begin{figure}
  \centering
  \includegraphics[width=.375\linewidth]{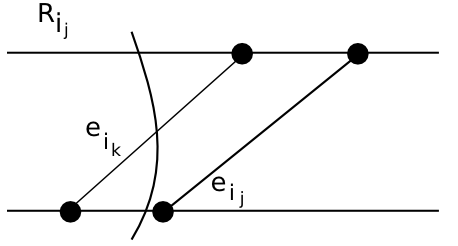}
  \captionof{figure}{Case (1)}
  \label{fig:5}

\end{figure}

Denote by  $A_j$  the set of vertices of type (1), and let
$\alpha_j=|A_j|$. The accumulating regain  in $G_j^a$ in this way
 is

\begin{equation}\label{regainontype1}
(s_j-1)(n-\mu_j-\ljm)\alpha_j
\end{equation}

$G_j^b$ gets  $n\alpha_j$, and since  $\lambda_j<n$ this means that it gets more than $\lambda_j$ for each vertex of this type, as promised.

\begin{figure}
\centering
\begin{minipage}{.5\textwidth}
  \centering
  \includegraphics[width=.75\linewidth]{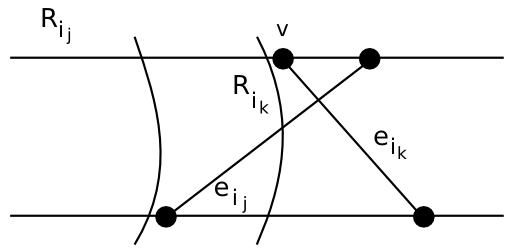}
  \captionof{figure}{Case 2a(i), one type of crossing}
  \label{fig:3}
\end{minipage}%
\begin{minipage}{.5\textwidth}
  \centering
  \includegraphics[width=.75\linewidth]{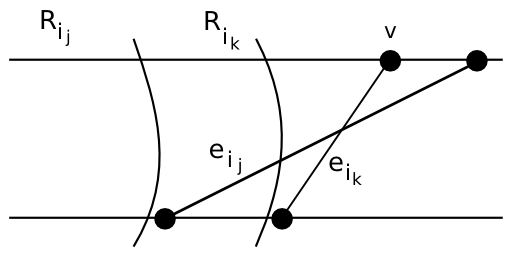}
    \captionof{figure}{Case 2a(i), another type of crossing}
  \label{fig:4}
\end{minipage}
\end{figure}

\item
$j< k(v)$.  This we divide into the following subcases:
\begin{enumerate}
\item
$e(v)$ is long and $k=k(v)<\omega$ and $\eij \not \subseteq \Rik$,  or $i(v)=i_{k}$ (see Notation
\ref{iv} for the definition of $i(v)$). The latter means that
$e(v)=\eik$.

 In this case  $k$ procrastinates with respect to $j$,
which entitles us to a regain of $s_kn$. Note that there
are at most $s_k-1$ vertices $v \in \skipmj$ that are tails of long edges, and satisfy $k(v)=k$.
So, distributing this gain among the vertices $v \in \skipmj$
that are tails of long edges, and satisfy $k(v)=k$, each gets at least a gain of $n$. Remembering that $\lambda_j<n$, we are fulfilling the requirement of ``$\lambda_j$ gain in $G^b_j$ for every vertex in $\skipmj$''.

The splitting of the gain between $G_j$ and $G_k$ is done in this
case according to a still finer classification into subcases:

\begin{enumerate}

\item $\eij$ and $\eik$ cross (see Definition \ref{crossing}.  We do not discern in this case
between the cases  $tail(\eik) \in M$ and $tail(\eik) \in W$ - see
Figures 4 and 5 for the two possibilities. We give a gain of $n-1$
to $G^b_j$, saving $1$ for a fine point below (see remark after case
(3)). By Observation \ref{lambdalessthann}  we are giving $G^b_j$ at
least $\lambda_j$, as required.
\item $\eij$ and $\eik$ are parallel.
Here $k$ is procrastinating with respect to $j$, and thus we are
entitled to a gain of $s_kn$. This is the same as Case (1), with the
roles of $j$ and $k$ reversed. This regain (that is shared between
stages $j$ and $k$) was considered in (1) for  stage $i_k$, and
hence we do not distribute regains for this case. But recall that
$G_j^b$ gets in stage $i_k$ its share of $(s_k-1)(\mu_k+\lambda_k)$
(keep in mind that the roles of the indices $j$ and $k$ are
reversed). By Lemma \ref{lambdasmall} this quantity is at least
$\lambda_j$ for each such vertex.\\
\end{enumerate}

\item $\eij \subseteq \rik$.  Then necessarily $k=\omega$ (see Figure 6).

With such vertices we associate a regain on double counting in the estimate $n|\rik|$ towards calculating
$|\fik|$, of all edges in $(M_j \cup \skipmj \cup \Tjm)\times
\{tail(\eij)\}$. The number of these edges is $\mu_j+\ljm+\tjm$. We give $G^b_j$ the amount of $h(j)\lambda_j$,
where $h(j)$ is the number of vertices in $\skipmj$ having $k(v)=\omega(j)$.

Note that no regain of this type is counted more than once. To see
this it is best to view the regain associated with each vertex  $v$
of this type as a regain on the calculation of $|F_{i(v)}|$ itself,
rather than using the inequality  $|F_{i(v)}| \le |F_{i_k(v)}|$.
Viewed this way, the sets of edges (which are actually stars) on
which there is double counting in $|F_{i(v)}|$ are disjoint for
different $v$'s. Note also that by
Lemma \ref{lambdalessthanmu} for each vertex of the present type we are adding at least $\lambda_j$ to $G^b_j$, as required. \\
\end{enumerate}

\item $k(v)=j$, meaning that $v=head(\eij)$.

On this vertex we have the same regain as on vertices of type (2b), with a gain of $\mu_j+\ljm+\Tjm$ given to $G^b_j$.

We have to be careful in this calculation, since in this case there
is danger of considering the double counting of an edge twice. Here
it may happen that for distinct $j_1$ and $j_2$  the vertices
$head(e_{i_{j_1}})$ and $head(e_{i_{j_2}})$  both represent the same
set of edges, namely $|F_{i_\omega}|$. Since one side in each edge
considered is $tail(\eij)$, this can happen only in one case: when
$\omega(j_1)=\omega(j_2)$ for indices $j_1 \neq j_2$, and that
$tail(e_{i_{j_1}})$ and $tail(e_{i_{j_2}})$ are on different sides.
In this case the double counting on the edge
$(tail(e_{i_{j_1}}),tail(e_{i_{j_2}}))$ is taken
 into account twice, while it should have been taken only once. In
this case we can compensate for this double-double counting in the
following way. Without loss of generality assume that $j_1 <j_2$.
Since $\omega(j_1)=\omega(j_2)$, the index $j_2$ procrastinates with
respect to $j_1$, which means that we are entitled to a gain of $n$
in the calculation of $|F_{i_{j_2}}|$, hence a gain of $s_{j_2}n$ in
the total sum. We only used $s_{j_2}\lambda_{j_1}$, and since
$s_{j_2} \ge 1$ and $\lambda_{j_1}<n$ (see Observation
\ref{lambdalessthann}), we have the desired compensation.
\end{enumerate}
\vspace{0.3cm}

Note that the regains given above to $G^b_j$ cover all of $L_j^b$.

\begin{figure}
  \centering
  \includegraphics[width=.375\linewidth]{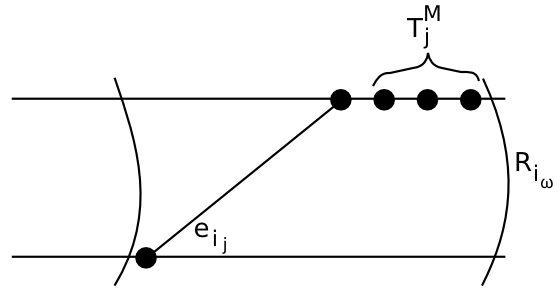}
  \captionof{figure}{Case 2(b)}
  \label{fig:5}

\end{figure}

\subsection{Another  regain on double counting}

We are entitled to another type of regain, on edges containing $A_j$
vertices. In the calculation of $|\fij|$ all edges between
$E_j[A_j]$ and $M \cap R_j$ are counted twice, so we are entitled to
a regain of $\mu_j\alpha_j$ on $|\fij|$, and thus of
$s_j\mu_j\alpha_j$ in total. In order to avoid considering this
double counting more than once we do not take into account vertices
contained in $j$-short edges - see Case (2b) above. Thus the regain
is $s_j(\mu_j-\sigma_j)\alpha_j$, which for ease of later
calculations we shall replace by the possibly smaller

 \begin{equation}\label{doublecountingAj} (s_j-1)(\mu_j-\sigma_j)\alpha_j
 \end{equation}

 \begin{example} To see why giving gains to the earlier indices is necessary, consider again Example \ref{stealexample}.
    There $L_1=q^2$, which is regained by a double count argument for $i_2$ (the second short edge).
In the baseline argument  $|F_{i_2}|$, and with it $|F_k|,~k \in S_2$, $k>2$ are estimated as $|R_{i_2}|n$. In this calculation all $q$ edges $m_cw_1$ in $R_{i_2}$ are double counted, so there is a gain of $q$ in the calculation of $|F_{i_2}|$, resulting in a gain of $s_2q=q^2$ in the baseline calculation - precisely $L_1$.

\end{example}

\section{Collecting the $G^a_j$ gains}

 \begin{lemma}\label{lowerboundongain}
$m-i_j \ge \ljm-\alpha_j$.
\end{lemma}

\begin{proof}
This follows from the fact that every vertex in ${SKIP}^M_j \setminus A_j$  is matched by some
edge $e_i,~~i \ge i_j$.
\end{proof}

 By the lemma and the definition of $G^{BASIC}_j$  (see \eqref{firstnetgain}) we have:

\begin{equation} \label{netgainfornow} G^{BASIC}_j \ge (s_j-1)( \lijm -\alpha_j)n \end{equation}

\begin{lemma}\label{nistheking}
$n \ge \ljw+\tmj+\sigma_j+\twj+\alpha_j$.
\end{lemma}

This follows from the fact that that $\alpha_j, \tjm, \tjw, \tljw$ and $\sigma_j$ are sizes of  disjoint subsets of $W$, namely $A_j, \skipjw,  E_j[\Tjm], \Tjw$ and $\bigcup \Sigma_j \cap W$.

The regain in \eqref{regainontype1}, $(s_j-1)(n-\mu_j-\tljm)\alpha_j$,
together with the regain of \eqref{doublecountingAj},  $(s_j-1)\alpha_j(\mu_j-\sigma_j)$,
and $G^{BASIC}_j$ sum up to

$$(s_j-1)(\tljm -
\alpha_j)n+(s_j-1)(n-\mu_j-\tljm)\alpha_j+(s_j-1)\alpha_j(\mu_j-\sigma_j)
$$

and we need to show that this sum, that is a lower bound for $G^a_j$,  is at least $L^a_j$. Namely, we have to show that

$$(s_j-1)(\tljm -
\alpha_j)n+(s_j-1)(n-\mu_j-\tljm)\alpha_j+(s_j-1)\alpha_j(\mu_j-\sigma_j) \ge (s_j-1)\ljm \lambda_j.
$$

Canceling out the term $s_j-1$ and additive terms, we need to prove:

$$\ljm(n - \alpha_j)-\sigma_j\alpha_j\ge \ljm\lambda_j$$

By Lemma \ref{nistheking} $\lambda_j \le n -\alpha_j -\sigma_j$.
Thus it is enough to show that $\ljm(n - \alpha_j)-\sigma_j\alpha_j
\ge \ljm(n-\alpha_j-\sigma_j)$, which follows from the fact that
$\alpha_j \le \ljm$ ($A_j$ being contained in $\skipmj$).
\\

This shows that $G^a_j \ge L^a_j$, thereby completing the proof of
Theorem \ref{maintheorem}.

\section{Proof of Theorem \ref{generalr3}}
Let $\cf$ be a collection of hypergraphs satisfying the condition of
the theorem. Order the vertices of the first side $V_1$ as  $v_1,
\ldots ,v_n$. By Lemma \ref{shifting} we may assume that all $F_i$
are shifted with respect to this order. Let $i_1$ be such that
$F_{i_1}$ has maximal degree at $v_1$ among all $F_i$'s. Then we
choose $i_2 \neq i_1$ for which $F_{i_2}$ has maximal degree at
$v_2$ among all $F_i,~~i \neq i_1$, and so forth. To save indices,
reorder the $F_i$'s so that  $i_j=j$ for all $j$. Let $H_j$ be the
set of $2$-edges incident with $v_j$ in $F_j$. It clearly suffices
to show that the collection $\ch=(H_j: j \le k)$ of subgraphs of
$K_{n,n}$ has a rainbow matching, so it suffices  to show that $\ch$
satisfies the conditions of Theorem \ref{maintheorem}. Assuming it
does not, since the sizes $|H_j|$ are descending, $\sum_{k-t <j \le
k} |H_j|= \sum_{k-t <j \le k} deg_{F_j}(v_j)\le t(t-1)n$ for some $t
< k$. We shall reach a contradiction to the assumption that
$|F_k|>(k-1)n^2$.\

 Write $m$ for  $|H_k|$. Clearly $$ \sum_{j \le k-t}deg_{F_k}(v_j) \le (k-t)n^2$$ and by the order by which $F_j$ were chosen
$$ \sum_{k-t <j \le k}deg_{F_k}(v_j) \le
\sum_{k-t <j \le k} deg_{F_j}(v_j)\le t(t-1)n$$
Since $\sum_{k-t <j \le k}deg_{F_k}(v_j) \ge mt$, this implies that $m \le n(t-1)$.

By the shifting property, $$\sum_{k <j \le n} deg_{F_k}(v_j)\le m(n-k)\le n(t-1)(n-k)$$
And so: $$\sum_{j>k-t}deg_{F_k}(v_j) \le t(t-1)n+(t-1)n(n-k)=n(t-1)(t+n-k)\le (t-1)n^2$$

Hence $$|F_k|=\sum_{j \le k}deg_{F_k}(v_j)\le (k-t)n^2+(t-1)n^2=(k-1)n^2$$
Which is the desired contradiction.


\section{Conjecture \ref{sizecondition} for large $n$}

\begin{theorem}
For every $r$ and $k$ there exists $n_0=n_0(r,k)$ such that Conjecture \ref{sizecondition} is true for all $n>n_0$.
\end{theorem}

\begin{proof}
By Lemma \ref{shifting} we may assume that all $F_i$'s are shifted.
Let $A_i$ consist of the first $k-1$ vertices in $V_i$~($i \le r$), and let $A =\bigcup_{i\le r}A_i$. Since the number of edges
meeting $A$ in two points or more is $O(n^{r-2})$, for large enough $n$
for each $i$ there exist at least $k-1$ points $x$ in $A$ such that $e \cap A=\{x\}$ for some $e \in F_i$. Hence we can choose edges
 $e_i \in F_i$ and distinct  points $x_i \in A$ ~$(i \le k-1)$ such that $e_i \cap A =\{x_i\}$.
 Since the number of edges going through $x_1,\ldots,x_{k-1}$ is no larger than $(k-1)n^{r-1}$,
 there exists an edge $e_k$ in $F_k$ missing $x_1,\ldots,x_{k-1}$. Using the shifting property,
  we can replace inductively each edge $e_i~, i\le k-1$, by an edge $e'_i\in F_i$ contained in
  $A$,  missing $e_k$ and missing all $e'_j,~j<i$. This yields a rainbow matching for $F_1, \ldots, F_k$.
\end{proof}

\section{Further conjectures}
Theorem \ref{mainbipartite} may be true also under the more general condition of degrees bounded by $n$.
\begin{conjecture}\label{degreecondition}
Let $d>1$, and let $F_1, \ldots, F_k$ be bipartite graphs on the same ground set,
satisfying $\Delta(F_i) \le d$ and $|F_i|>(k-1)d$. Then the system $F_1, \ldots, F_k$ has a rainbow matching.
\end{conjecture}

For $d=1$ this is false, since for every $k>1$ there are matchings
$F_1, \ldots, F_k$ of size $k$ not having a rainbow matching.

Theorem \ref{maintheorem} has a simpler counterpart, which we believe to be true:

\begin{conjecture}\label{simple}
If $F_i,~i \le k$ are subgraphs of $K_{n,n}$ satisfying $|F_i|\ge
in$ for all $i \le k$, then they have  a rainbow matching.
\end{conjecture}

\begin{theorem}
Conjecture  \ref{simple} is true for $n > \binom{k}{2}$.
\end{theorem}

\begin{proof}
As before, we assume that $F_i$ are all shifted. Number one side of
$K_{n,n}$ as $w_1, \ldots,w_n$. Let $d_{i,j}=deg_{F_{i}}(w_j)$. Let
$M$ be a $k \times k$ $0,1$ matrix, defined by: $m_{i,j}=1$ if
$d_{i,j}> k-j$ and $m_{i,j}=0$ otherwise. It is enough to find a
permutation $\pi: [k] \to [k]$ such that $m_{\pi(j),j}=1$ for all
$j$, since then one can match the vertices $w_j$ in $F_{\pi(j)}$
greedily, one by one, starting at $w_k$: at the $j$-th step, when $w_k, \ldots,w_{k-j+1}$
have already been matched, since $deg_{F_{\pi(j-k)}}(w_{k-j})\ge j$ there
exists at least one edge in $F_{\pi(k-j)}$ incident with $w_{k-j}$ that
can be added to the rainbow matching.

Assuming that there is no such permutation $\pi$, by Hall's theorem
there is a set $J$ of $p$ columns of $M$ and a set $I$ of $k-p+1$
rows, such that $m_{i,j}=0$ for all $i \in I,~j \in J$. Let $q$ be
the largest element of $I$. Then $q \ge k-p+1$. We shall show that $|F_q|<n(k-p+1)$, contradicting the assumption of the conjecture.

Let  $J=\{j_1,j_2,\ldots,j_p\}$, arranged in ascending order. Since $q \in I$, we have $d_{q,j_s}\le k-j_s$ for all $s \le p$. Since the sequence $d_{q,j}$ is non increasing in $j$,
we have:

\begin{equation}\label{js}
|F_q|=\sum_{j \le n}d_{q,j} \le n(j_1-1)+(j_2-j_1)(k-j_1)+(j_3-j_2)(k-j_2)+\ldots+(j_{p-1}-j_p)(k-j_{p-1})+(n-j_p+1)(k-j_p)
\end{equation}

Call the right hand side of (\ref{js}) $c(J)$. Suppose that there exists $s<p$ such that $j_{s}+1<j_{s+1}$. Then, moving $j_s$ to the right, namely replacing $j_{s}$ in $J$ by $j_{s}+1$, decreases $c(J)$
by $1$ (the decrease in the term corresponding to $j_s$) and increases by $j_{s+1}-j_s-1$
(corresponding to the increase in the terms between $j_{s}+1$ and $j_{s+1}$). This means that $c(J)$ has not decreased. Hence, writing $j$ for $j_p$, we have:
\begin{equation}\label{js2} c(J) \le  c(\{j-p+1,j-p+2,\ldots,j\})
\end{equation}
Writing $\gamma(j)$ for the right hand side of (\ref{js2}), we have: $$\gamma(j)=n(j-p)+(k-j+p-1)+(k-j+p+2)+\ldots +(k-j)+(n-j)(k-j)=$$
$${p\choose 2}+p(k-j)+n(j-p)+(n-j)(k-j)$$

This is a quadratic expression in $j$, which attains its maximum at one of the two extremes, $j=p$ or $j=k$.
In fact,  for both values of $j$ it attains the same value, ${p\choose 2}+n(k-p)$.
We have shown that $|F_q|<{p\choose 2}+n(k-p)$.
By the assumption $n>{k\choose 2}$ this implies that $|F_q|<n(k-p+1)$, which is the desired contradiction.

\end{proof}

To formulate yet another conjecture we shall use the following
notation:

\begin{notation}\hfill
\begin{enumerate}
\item
For a sequence $a=(a_i, ~1 \le i \le k)$ of real numbers we denote by $\overrightarrow{a}$ the sequence rearranged in non-decreasing order.
\item
Given two sequences $a$ and $b$ of the same length $k$, we write $a\le b$ (respectively $a <b$) if $\overrightarrow{a}_i \le \overrightarrow{b}_i$ (respectively $\overrightarrow{a}_i < \overrightarrow{b}_i$) for all $i \le k$.
\end{enumerate}
\end{notation}

Given subgraphs $F_i, ~~i \le k$ of $K_{n,n}$, define a $k \times n$
matrix $A=(a_{ij})$ as follows.  Order one side of the bipartite
graph as $v_1, v_2, \ldots, v_n$, and let $a_{ij}=deg_{F_i}(v_j)$.
The $i$-th row sum $r_i(A)$ of $A$ is then $|F_i|$. Thus, Theorem
\ref{maintheorem} can be formulated as follows:
\begin{theorem} If $\sum_{i \le j}\overrightarrow{r}_i >j(j-1)n$ for every $j \le k$ then there exists a permutation $\pi: [k] \to [k]$ such that $a_{i\pi(i)} \ge (1,2, \ldots,k)$.
\end{theorem}

We believe that the following stronger conjecture is true:

\begin{conjecture} If $\sum_{i \le j}\overrightarrow{r}_i >j(j-1)n$ for every $j \le k$ then there exists a permutation $\pi: [k] \to [k]$ such that
$\sum_{i \le j}\overrightarrow{a}_{i\pi(i)} >j(j-1)$ for every $j$.
\end{conjecture}

{\bf Acknowledgements:} We are grateful to  Roy Meshulam for the
proof of Theorem \ref{af} and for pointing out to us the relevance
of the shifting method. We are also grateful to Zoltan F\"uredi and
Ron Holzman for helpful information, and to Eli Berger for
stimulating discussions.

\end{document}